\documentclass[12pt,a4paper]{amsart}
\usepackage{amsfonts,amsmath,amssymb,amscd}
\usepackage{fancybox}
\usepackage[latin1]{inputenc}
\usepackage[usenames]{color}
\usepackage{graphicx}
\usepackage{enumerate}
\usepackage{subfig}
\usepackage{float}
\usepackage{fullpage}
\usepackage{epsfig,amssymb}

\theoremstyle{plain}
\newtheorem{thm}{Theorem}[section]
\newtheorem{cor}[thm]{Corollary}
\newtheorem{lem}[thm]{Lemma}
\newtheorem{prop}[thm]{Proposition}

\theoremstyle{definition}

\newtheorem{rem}[thm]{Remark}

\def \r{\mbox{${\mathbb R}$}}

\DeclareMathOperator{\trace}{trace}

\begin{document}
	
\title{Triharmonic curves in 3-dimensional homogeneous spaces}

\author{S. Montaldo}
\address{Universit\`a degli Studi di Cagliari\\
Dipartimento di Matematica e Informatica\\
Via Ospedale 72\\
09124 Cagliari, Italia}
\email{montaldo@unica.it}

\author{A. P\'ampano}
\address{Department of Mathematics\\
University of the Basque Country\\
Aptdo. 644\\
48080, Bilbao, Spain.}
\email{alvaro.pampano@ehu.eus}
\date{\today}

\begin{abstract}
We first prove that, unlike the biharmonic case, there exist triharmonic curves with nonconstant curvature in a suitable Riemannian manifold of arbitrary dimension. We then give the complete classification of triharmonic curves in surfaces with constant Gaussian curvature. Next, restricting to curves in a 3-dimensional Riemannian manifold, we study the family of triharmonic curves with constant curvature, showing that they are Frenet helices.
In the last part, we give the full classification of triharmonic Frenet helices in space forms and in Bianchi-Cartan-Vranceanu spaces.
\end{abstract}

\thanks{Work partially supported by Fondazione di Sardegna (Project STAGE) and Regione Autonoma della Sardegna (Project KASBA). The second author has been partially supported by MINECO-FEDER grant PGC2018-098409-B-100, Gobierno Vasco grant IT1094-16 and Programa Posdoctoral del Gobierno Vasco, 2018. He also wants to thank the Department of Mathematics and Computer Science of the University of Cagliari for the warm hospitality during his stay.}

\maketitle

\section{Introduction}

An arc-length parametrized curve $\gamma: I\rightarrow M^n$ from an open interval $I\subset\mathbb{R}$ to a Riemannian manifold of dimension $n$ is called \emph{triharmonic} if
\begin{equation}
\nabla_T^5 T+R^M(\nabla_T^3 T,T)\,T-R^M(\nabla_T^2 T, \nabla_T T)\,T=0 \nonumber
\end{equation}
where $T$ is the unit tangent vector field of $\gamma$, $\nabla$ denotes the Levi-Civita connection of $M^n$ and $R^M$ is the Riemannian curvature tensor of $M^n$.

Triharmonic curves represent the case $r=3$ in a general theory of $r$-harmonic (polyharmonic) curves. The theory of these curves can be considered as the 1-dimensional case of $r$-harmonic maps, first introduced in \cite{ES2}, where Eells and Sampson, soon after their celebrated paper on harmonic mapping \cite{ES}, suggested the idea of studying critical points of higher order energies as a possible generalization of harmonic maps. For an updated account on  higher order energies we recommend the interested reader to see \cite{BMRO}. 

The case $r=2$, that is of biharmonic curves, is well studied and it is well known (see, for example, \cite{CMOP}) that if we denote by $\kappa(s)=\lVert\nabla_T T\rVert$  the \emph{curvature} of an arc-length parametrized curve $\gamma:I\rightarrow M^n$ in a Riemannian manifold $M^n$, 
then if $\gamma$ is proper biharmonic the curvature $\kappa$ is constant. 

In the first part of the paper we investigate the possibility of constructing triharmonic curves in a Riemannian manifold with nonconstant curvature and we obtain the following result.

\begin{thm}\label{existence} For any $n>1$ there exist a triharmonic curve with nonconstant curvature in $S\times\mathbb{R}^{n-2}$ where $S$ is locally a ruled surface in $\mathbb{R}^3$, parametrized by $x(s,t)=\alpha(s)+t N(s)$, and $\alpha(s)$ is, up to rigid motions, the only curve in $\mathbb{R}^3$ with curvature and torsion given by  
$$
\kappa(s)=\frac{\sqrt{5}}{s}\,,\quad \tau(s)=\frac{3\sqrt{7}}{2s}\,.
$$
\end{thm}

Theorem~\ref{existence} is achieved by an analysis of triharmonic curves parametrized by arc-length in a surface $S$. This analysis also permits us to give the classification of triharmonic curves in surfaces with constant Gaussian curvature (Theorem~\ref{all}).

In the next part, we shall investigate triharmonic curves in a Riemannian manifold of dimension 3. In this case the general study of triharmonic curves is more complicated and we shall restrict ourselves to the study of triharmonic curves with constant curvature. We first prove that triharmonic curves with constant curvature  in a Riemannian manifold of dimension 3 are Frenet helices (Corollary~\ref{helices}). The latter result enables us to tackle the classification problem of triharmonic curves with constant curvature in homogeneous 3-dimensional manifolds. 

We recall that, among homogeneous 3-dimensional manifolds, there are the $3$-space forms $M^3(\rho)$ when the isometry group is of maximal dimension, that is 6. Triharmonic curves with constant curvature in space forms were studied by Maeta in \cite{M}. Here, we recover Maeta's result and we slightly improve on it by showing, in Proposition~\ref{pro:constantktau}, that for a triharmonic curve  the torsion is constant if and only if the curvature is constant and, consequently, Maeta's examples are the only triharmonic curves with either constant curvature or constant torsion.

On the other hand, homogeneous 3-dimensional manifolds with the isometry group of dimension $4$ can be locally described as Bianchi-Cartan-Vranceanu spaces $M(a,b)$ with $4a\neq b^2$. Similar to what happens in 3-dimensional space forms $M^3(\rho)$, if the torsion of a triharmonic curve is identically zero, then we prove that its curvature is constant. We then classify triharmonic curves with zero torsion in Theorem~\ref{46}.

Finally, in Theorems~\ref{others} and \ref{h3} we give the full classification and their explicit parametrizations of triharmonic helices in Bianchi-Cartan-Vranceanu spaces $M(a,b)$ with $4a\neq b^2$. It turns out that these triharmonic curves can be seen as geodesics of suitable Hopf cylinders (see Corollary~\ref{hopf}).

\section{Triharmonic curves in riemannian manifolds}

\emph{Harmonic maps} $\varphi\,: (\widetilde{M},h)\rightarrow (M,g)$ between Riemannian manifolds are the critical points of the energy functional
\begin{equation}
E(\varphi)=\frac{1}{2}\int_{\widetilde{M}} \lVert d\varphi \rVert ^2 v_h\, . \nonumber
\end{equation}
The corresponding Euler-Lagrange equation is given by the vanishing of the tension field
\begin{equation}
\tau(\varphi)=-d^*d\varphi=\trace \nabla d\varphi\,, \nonumber
\end{equation}
where $d$ is the exterior differentiation and $d^*$ is the codifferentiation.
In \cite{ES2}, Eells and Sampson suggested to study \emph{$r$-harmonic maps} (or simply, \emph{polyharmonic maps}) as the critical points of the \emph{$r$-energy functional} defined by
\begin{equation}\label{eq:ESenergy}
E_r^{ES}(\varphi)=\frac{1}{2}\int_{\widetilde{M}} \lVert (d+d^*)^r \varphi\rVert^2 v_h\, , \quad r\geq 1,
\end{equation}
for $\varphi\in\mathcal{C}^\infty (\widetilde{M},M)$. 
When the dimension of $\widetilde{M}$ is one, the $r$-energy functional \eqref{eq:ESenergy} coincides with 
another higher order energy functional,  first studied by Wang in \cite{W}  and by Maeta in \cite{Maeta1}, that, when $r=2s+1$, $s \geq 1$, takes the form
\begin{equation}\label{eq:wangenergy}
E_{2s+1}(\varphi)= \frac{1}{2} \int_{\widetilde{M}} \, \langle\,d\underbrace{(d^* d) \ldots (d^* d)}_{s\, {\rm times}}\varphi, \,d\underbrace{(d^* d) \ldots (d^* d)}_{s\, {\rm times}}\varphi\,\rangle\, v_h\, ,
\end{equation}
while if $r=2s$, $s \geq 1$, is
\begin{eqnarray}\label{2s-energia}
E_{2s}(\varphi)&=& \frac{1}{2} \int_{\widetilde{M}} \, \langle \, \underbrace{(d^* d) \ldots (d^* d)}_{s\, {\rm times}}\varphi, \,\underbrace{(d^* d) \ldots (d^* d)}_{s\, {\rm times}}\varphi \, \rangle\, v_h \,.
\end{eqnarray}

For a complete description of the relations between the functional \eqref{eq:ESenergy} and the functionals \eqref{eq:wangenergy} and \eqref{2s-energia} we refer the reader to \cite{BMRO}. 

It follows that, when $\gamma: I\rightarrow M$ is a curve parametrized by arc-length, from an open interval $I\subset\mathbb{R}$ to a Riemannian manifold, putting $\gamma' =T$, the Euler-Lagrange equations of \eqref{eq:wangenergy} and \eqref{2s-energia}, computed by Wang,  reduces to the  equation  
\begin{equation}\label{eq:rharmoniccurveEL}
\tau_r(\gamma)=\nabla_{T}^{2r-1} T+\sum_{\ell=0}^{r-2} (-1)^{\ell} R^M\left(\nabla_T^{2r-3-\ell} T, \nabla_T ^{\ell} T\right) \,T=0\, ,\quad r\geq 1\,.
\end{equation}
Solutions of \eqref{eq:rharmoniccurveEL} are called $r$-harmonic curves. In particular, any harmonic curve is a polyharmonic curve, for any $r\geq 1$. We say that a $r$-harmonic curve is \emph{proper} if it is not harmonic. Therefore, the main interest is to find and classify proper $r$-harmonic curves.

Throughout this paper, we shall focus on \emph{triharmonic curves} (polyharmonic curves for $r=3$), which are the arc-length parametrized curves solutions of \eqref{eq:rharmoniccurveEL} for $r=3$, that is solutions of the following equation
\begin{equation}
\tau_3(\gamma)=\nabla_T^5 T+R^M(\nabla_T^3 T, T)\,T-R^M(\nabla_T^2T,\nabla_T T)\,T=0\,. \label{3-tension}
\end{equation}

Notice that, as mentioned above, every harmonic curve  is a triharmonic curve. However, as proved by Maeta in \cite{M},  biharmonic curves (polyharmonic curves for $r=2$) are not necessary triharmonic curves and, viceversa, triharmonic curves do not need to be biharmonic. Thus the study of triharmonic curves could be, in general, a completely different problem to that of biharmonic curves.

\section{Triharmonic curves in a surface}

We begin by proving the existence of a surface $S$ in $\r^3$ admitting proper triharmonic curves with nonconstant curvature. We shall denote the metric on $S$ by $\langle,\rangle$.

Let $\gamma(s)$ be an arc-length parametrized curve immersed in a surface $S$. The vector field $T=\gamma'$ is the unit tangent to $\gamma$, while we denote by $N_S=JT$ its unit normal. Here, $J$ is the counter-clockwise rotation by an angle $\pi/2$ defined in the tangent bundle of $S$. Then, if $\nabla$ denotes the Levi-Civita connection of $S$, the following Frenet-type equation holds
\begin{equation}\label{FTS}
\nabla_T T=\kappa_g(s) N_S\,,
\end{equation}
where $\kappa_g(s)$ is the \emph{geodesic curvature} of $\gamma$.

Next, looking at the tangent and normal components of \eqref{3-tension}, we obtain the following characterization of triharmonic curves in surfaces.

\begin{prop} An arc-length parametrized curve $\gamma(s)$ immersed in a surface $S$ is a triharmonic curve if and only if its geodesic curvature is a solution of the following system of differential equations
\begin{eqnarray}
\kappa_g\kappa'''_g +2\kappa'_g\kappa''_g-2\kappa^3_g\kappa'_g&=&0\,,\label{eq1}\\
\kappa^{(4)}_g-15\kappa_g\left(\kappa'_g\right)^2-10\kappa_g^2\kappa''_g+\kappa_g^5+K_S\left(\kappa''_g-2\kappa_g^3\right)&=&0\,.\label{eq2}
\end{eqnarray}
where $K_S=\langle R^S\left(T,N_S\right)N_S,T\rangle$ is the Gaussian curvature of $S$ along $\gamma$. Here, $\left(\,\right)'$ denotes the derivative with respect to the arc-length parameter $s$.
\end{prop}
\begin{proof} Triharmonic curves on $S$ are the arc-length parametrized curves solutions of \eqref{3-tension}. Applying \eqref{FTS} as many times as needed and after a long straightforward computation we obtain that \eqref{3-tension} can be written as
\begin{eqnarray*}
-5\left(\kappa_g\kappa_g'''+2\kappa_g'\kappa_g''-2\kappa_g^3\kappa_g'\right)T+&&\\
\left(\kappa^{(4)}_g-15\kappa_g\left[\kappa'_g\right]^2-10\kappa_g^2\kappa''_g+\kappa_g^5+K_S\left[\kappa''_g-2\kappa_g^3\right]\right)N_S&=&0\,,
\end{eqnarray*}
obtaining the desired result.
\end{proof}

We now proceed with the construction of a surface $S$ admitting triharmonic curves with nonconstant curvature.

Let $\alpha(s)$ be an arc-length parametrized curve of $\mathbb{R}^3$ with curvature  given by $\kappa(s)=\lVert\alpha''(s)\rVert$. If $\kappa(s)\neq 0$, that is, if $\alpha(s)$ is not a line, then the torsion of $\alpha(s)$ is
$$\tau(s)=\frac{\det\left(\alpha',\alpha'',\alpha'''\right)}{\lVert \alpha'\times\alpha''\rVert^2}=\frac{\det\left(\alpha',\alpha'',\alpha'''\right)}{\kappa^2(s)}\,,$$
where $\times$ denotes the usual vector product. Do not confuse the notation with the tension field $\tau(\varphi)$ defined in \S 2. For a non-linear curve $\alpha(s)$ in $\mathbb{R}^3$, we denote the usual Frenet frame along $\alpha$ by
$$\{T(s)=\alpha'(s), N(s)=\alpha''(s)/\kappa(s),B(s)=T(s)\times N(s)\}\,$$
where $N$ and $B$ are the unit normal and unit binormal to $\alpha$, respectively.

Consider now the ruled surface $S$ immersed in $\mathbb{R}^3$ defined by the local parametrization $x(s,t)=\alpha(s)+t N(s)$. Then, the Gaussian curvature along $\alpha(s)$ satisfies $K_S\left(\alpha(s)\right)=-\tau^2(s)$. Note also that the geodesic curvature of $\alpha(s)$, as a curve in $S$,  is, up to a sign, $\kappa(s)$, the curvature of $\alpha(s)$ as a curve in $\mathbb{R}^3$. Thus, after a change of orientation in $S$, if necessary, we can assume without loss of generality that $\kappa_g(s)=\kappa(s)$. Using this in equations \eqref{eq1}-\eqref{eq2}, we have that $\alpha(s)$ is a triharmonic curve in $S$ if and only if
\begin{eqnarray}
\kappa\kappa''' +2\kappa'\kappa''-2\kappa^3\kappa'&=&0\,,\label{eq11}\\
\kappa^{(4)}-15\kappa\left(\kappa'\right)^2-10\kappa^2\kappa''+\kappa^5-\tau^2\left(\kappa''-2\kappa^3\right)&=&0\,.\label{eq22}
\end{eqnarray}
Equation \eqref{eq11} only depends on the curvature $\kappa(s)$ while \eqref{eq22} depends on both $\kappa(s)$ and $\tau(s)$. Therefore, if there exists a nonconstant solution of \eqref{eq11} such that $\kappa''(s)\neq2\kappa^3(s)$, we can define a suitable torsion $\tau(s)$ as a solution of \eqref{eq22}. 

We recall that, by the Fundamental Theorem of Curves, an arc-length parametrized curve in $\mathbb{R}^3$ is completely determined, up to rigid motions, by its curvature and torsion. As a consequence, the nonconstant curvature $\kappa(s)$ which is a solution of \eqref{eq11} and the suitable election for the torsion $\tau(s)$, so that equation \eqref{eq22} is satisfied, completely determine the curve $\alpha(s)$ and, consequently,  the surface $S$. Moreover, $\alpha(s)$ shall be a triharmonic curve in $S$ with nonconstant (geodesic) curvature $\kappa_g(s)=\kappa(s)$.

In order to obtain solutions of \eqref{eq11} we follow \cite{K}. Assume that $\kappa(s)\neq 0$ and multiply \eqref{eq11} by $\kappa$. This makes the first two terms an exact derivative. At the same time, the last term is clearly a derivative and, hence, we can integrate once obtaining
\begin{equation}\label{help}
5\kappa^2\kappa''-2\kappa^5=c_1
\end{equation}
for some real constant $c_1$. Next, since we are seeking nonconstant solutions we assume that $\kappa'(s)\neq 0$ and multiply \eqref{help} by $2\kappa'\kappa^{-2}$. After this multiplication, we obtain an exact equation whose first integral is
\begin{equation}\label{fi}
5\left(\kappa'\right)^2=c_2-2c_1\frac{1}{\kappa}+\kappa^4
\end{equation}
for another real constant $c_2$.

Equation \eqref{fi} represents a biparametric family of first order differential equations in separable variables. Therefore, the family of solutions depends on three parameters. However, the last of these parameters can be omitted after translating the origin of the arc-length parameter $s$, if necessary.

In order to have an explicit solution, we consider the simplest possible case, $c_1=c_2=0$. In this case, equation \eqref{fi} can easily be solved obtaining that
\begin{equation}\label{curvature}
\kappa(s)=\frac{\sqrt{5}}{s}\,.
\end{equation}
It turns out that the function $\kappa(s)$ given in \eqref{curvature} satisfies
$$\kappa''(s)-2\kappa^3(s)=-\frac{8\sqrt{5}}{s^3}\neq 0$$
and, hence, as mentioned above, we can obtain a function $\tau(s)$ so that equation~\eqref{eq22} also holds. After some simplifications, we get
\begin{equation}\label{torsion}
\tau(s)=\frac{3\sqrt{7}}{2s}\,.
\end{equation}

\begin{rem} The curve $\alpha(s)$ in $\mathbb{R}^3$ uniquely  determined (up to rigid motions) by the curvature and torsion given by \eqref{curvature} and \eqref{torsion} respectively, satisfies the relation
$$\tau(s)=\frac{3}{2}\sqrt{\frac{7}{5}}\,\kappa(s)\,.$$
Curves satisfying a relation of the type $\tau(s)=\lambda\kappa(s)$, $\lambda\in\mathbb{R}$, are known in the literature as \emph{Lancret curves}, i.e. they are curves making a constant angle with a fixed direction, \cite{B}.
\end{rem}

We point out that for any choices of constants $c_1$ and $c_2$ in \eqref{fi}, the solution of that equation defines a nonconstant curvature, $\kappa(s)$. Moreover, equation~\eqref{eq22} always defines a torsion $\tau(s)$. In fact, using \eqref{help}, we can see that if $\kappa(s)$ is not constant, then $\kappa''\neq 2\kappa^3$ always holds. 

In conclusion, we can summarize the above discussion in the following proposition.

\begin{prop}\label{pro:existence}
Let $\alpha(s)$ be an arc-length parametrized curve of $\mathbb{R}^3$ with nonconstant curvature $\kappa(s)$ which is a solution of \eqref{fi} and nonconstant torsion $\tau(s)$ given by \eqref{eq22}. Let  $S$ be the ruled surface in $\r^3$ locally parametrized by $x(s,t)=\alpha(s)+t N(s)$. Then, $\alpha(s)$ is a triharmonic curve in $S$ with nonconstant geodesic curvature $\kappa_g(s)=\kappa(s)$.
\end{prop}

\subsection{Proof of Theorem~\ref{existence}}

Now, using Proposition~\ref{pro:existence},  we are going to prove Theorem~\ref{existence}. Let $\alpha(s)$ be the unique (up to rigid motions) curve parametrized by arc-length in $\r^3$  whose  curvature and torsion are given by \eqref{curvature} and \eqref{torsion}, respectively.  Let $S$ be the surface in $\r^3$ locally parametrized by $x(s,t)=\alpha(s)+t N(s)$ and denote by $i:S\hookrightarrow S\times\mathbb{R}^{n-2}$ the canonical inclusion of $S$ in the product space $S\times\mathbb{R}^{n-2}$ (of dimension $n$) defined by $i(p)=\left(p,{\bf 0}\right)$ for any $p\in S$. Then, it is a straightforward computation to check that $i\left(\alpha(s)\right)$ is a triharmonic curve in $S\times\mathbb{R}^{n-2}$ with nonconstant curvature $\kappa(s)=\lVert\nabla_T T\rVert$ given in \eqref{curvature}. This finishes the proof of Theorem \ref{existence}.

\subsection{Triharmonic curves in 2-dimensional space forms}

We now consider  triharmonic curves with constant geodesic curvature immersed in a surface. Clearly, as mentioned above, if $\gamma(s)$ is a geodesic of $S$, that is, if its geodesic curvature vanishes identically then equations \eqref{eq1}-\eqref{eq2} are trivially satisfied. On the other hand, if $\gamma(s)$ has non-vanishing constant geodesic curvature then it is triharmonic if and only if along $\gamma$
$$K_S=\frac{1}{2}\kappa_g^2\,.$$
In particular, the Gaussian curvature, $K_S$, along a proper triharmonic curve $\gamma$ must be a positive constant. 

We then assume that the surfaces $S$ has positive constant Gaussian curvature, $K_S=\rho> 0$. These surfaces are locally isometric to the sphere $\mathbb{S}^2(\rho)$. In this case, as first proved by Maeta in \cite[Corollary 5.3]{M}, circles satisfying $\kappa_g^2=2\rho$ are proper triharmonic curves. It turns out that these are all the proper triharmonic curves in surfaces with constant Gaussian curvature as proved in the following theorem.

\begin{thm}\label{all} Let $S$ be a surface with constant Gaussian curvature $K_S$ and let $\gamma(s)$ be a triharmonic curve in $S$ with geodesic curvature $\kappa_g$. If $K_S\leq 0$, then $\gamma(s)$ is a geodesic. On the other hand, if $K_S>0$, $\gamma(s)$ is either a geodesic or a circle satisfying $\kappa_g^2=2K_S$. 
\end{thm}
\begin{proof} We consider first the case where the geodesic curvature is constant. As argued above, if $\gamma(s)$ is a triharmonic curve in $S$ with constant geodesic curvature $\kappa_g$, then either $\gamma(s)$ is a geodesic ($\kappa_g=0$) or $2K_S=\kappa_g^2$ holds. Clearly, the latter is only possible whenever $K_S>0$.\\
Next, we are going to prove that there are no triharmonic curves in $S$ with nonconstant geodesic curvature. Assume that $\gamma(s)$ is a triharmonic curve with nonconstant geodesic curvature, $\kappa_g(s)$. Then, following \cite{K} again, equation~\eqref{eq1} can be integrated, as we have done for  \eqref{help}, obtaining 
\begin{equation}\label{help1}
\kappa_g''=\frac{c_1}{5\kappa_g^2}+\frac{2}{5}\kappa_g^3
\end{equation}
and, as for \eqref{fi},
\begin{equation}\label{help2}
\left(\kappa_g'\right)^2=\frac{c_2}{5}-\frac{2c_1}{5\kappa_g}+\frac{1}{5}\kappa_g^4\,,
\end{equation}
where $c_1$ and $c_2$ are real constants. Since $\gamma(s)$ is triharmonic, equations \eqref{eq1} and \eqref{eq2} must be satisfied simultaneously. We now differentiate \eqref{eq1} and combine with \eqref{eq2} to eliminate the term $\kappa_g^{(4)}$. Then, with the aid of \eqref{help1} and \eqref{help2} we obtain, after a long but straightforward computation, the following polynomial equation of degree ten in $\kappa_g$
$$
51\kappa_g^{10}+75\kappa_g^9+40K_S\kappa_g^8+63c_2\kappa_g^6-84c_1\kappa_g^5-5c_1K_S\kappa_g^3-6c_1c_2\kappa_g+14c_1^2=0\,.$$
Thus $\kappa_g$ must be constant, which contradicts the assumption that $\gamma(s)$ is a triharmonic curve with nonconstant geodesic curvature. This concludes the proof.
\end{proof}

Surfaces with constant Gaussian curvature are locally isometric to 2-dimensional space forms $M^2(\rho)$, that is the Euclidean plane $\mathbb{R}^2$ if $\rho=0$; the round 2-sphere $\mathbb{S}^2(\rho)$ if $\rho>0$;  the hyperbolic plane $\mathbb{H}^2(\rho)$ if $\rho<0$. Then, interpreting Theorem \ref{all} to 2-dimensional space forms we obtain 

\begin{cor}\label{M2rho} Let $M^2(\rho)$ be a 2-dimensional Riemannian space form. If $\rho\leq 0$, the only triharmonic curves are geodesics. If $\rho>0$, triharmonic curves are either geodesics or circles satisfying $\kappa_g^2=2\rho$.
\end{cor}

\section{Triharmonic helices in homogeneous 3-dimensional spaces}

In this section we are going to study proper triharmonic curves with constant curvature in a Riemannian manifold $M^3$ of dimension $3$.

Let us denote by $\gamma(s)$ an arc-length parametrized curve immersed in $M^3$ and let's put $\gamma'(s)=T(s)$. Assume that $\gamma(s)$ is non-geodesic, then $\gamma(s)$ is a Frenet curve of rank $2$ or $3$ and the standard Frenet frame along $\gamma(s)$ is denoted by $\{T(s),N(s),B(s)\}$. The Frenet equations are
\begin{eqnarray}\label{Frenetequations}
\begin{cases}
\nabla_T T(s)=\kappa(s)N(s)\\
\nabla_T N(s)=-\kappa(s)T(s)+\tau(s)B(s)\\
\nabla_T B(s)=-\tau(s)N(s)
\end{cases}
\end{eqnarray}
where $\kappa(s)$ is the curvature of $\gamma(s)$, while the function $\tau(s)$ is the \emph{torsion} of $\gamma(s)$. We shall say that a curve is a \emph{Frenet helix} if both  $\kappa(s)$ and $\tau(s)$ are constant.

By using equations \eqref{Frenetequations} in the equation $\tau_3(\gamma)=0$, \eqref{3-tension}, we can obtain a system of three differential equations characterizing triharmonic curves in $M^3$. Each of those differential equations corresponds to the tangent, normal and binormal component of the vector equation \eqref{3-tension}. In particular, the tangent component yields immediately the following result.

\begin{prop}\label{propeq13space} Let $\gamma(s)$ be an arc-length parametrized (proper) triharmonic curve immersed in a $3$-dimensional Riemannian manifold $M^3$, then
\begin{equation}\label{eq13space}
2\,\frac{d}{ds}\left(\kappa^2\kappa''\right)=\kappa\,\frac{d}{ds}\left(\kappa^2\left[\kappa^2+\tau^2\right]\right),
\end{equation}
where $\kappa=\kappa(s)$ and $\tau=\tau(s)$ are the curvature and torsion of $\gamma(s)$, respectively.
\end{prop}

As a consequence of Proposition \ref{propeq13space}, we conclude with the following characterization  of proper triharmonic curves in $M^3$ with constant curvature.

\begin{cor}\label{helices} Let $\gamma(s)$ be a proper triharmonic curve immersed in a $3$-dimensional Riemannian manifold $M^3$ with constant curvature $\kappa(s)=\kappa_o$. Then the curve $\gamma(s)$ is a Frenet helix. Moreover, the curvature $\kappa_o\neq 0$ and the torsion $\tau_o$ satisfy the system 
\begin{eqnarray}
\left(\kappa_o^2+\tau_o^2\right)^2-\left(2\kappa_o^2+\tau_o^2\right)\langle R^M\left(N,T\right)T,N\rangle-\kappa_o\tau_o\langle R^M\left(B,N\right)T,N\rangle&=&0\,,\label{eq1helix}\\
{\left(2\kappa_o^2+\tau_o^2\right)\langle R^M\left(N,T\right)T,B\rangle+\kappa_o}\tau_o\langle R^M\left(B,N\right)T, B\rangle&=&0\,.\label{eq2helix}
\end{eqnarray}
\end{cor}
\begin{proof} Since $\gamma(s)$ is a proper triharmonic curve, its curvature and torsion satisfy  \eqref{eq13space}, which implies, since the curvature $\kappa(s)=\kappa_o$ is a nonzero constant, that the torsion is necessarily constant, proving that the curve is a Frenet helix. Finally, assuming that  the curvature $\kappa_o\neq 0$ and the torsion $\tau_o$ are constant, the normal component and the binormal component of \eqref{3-tension} become, after  a long but straightforward computation, \eqref{eq1helix} and \eqref{eq2helix}, respectively.
\end{proof}

\subsection{Triharmonic helices in homogeneous 3-dimensional manifolds}

From now on, we are going to restrict ourselves to the analysis of proper triharmonic helices in homogeneous 3-dimensional manifolds.

A Riemannian manifold $M^n$ is said to be \emph{homogeneous} if for every two points $p$ and $q$ in $M^n$, there exists an isometry of $M^n$ mapping $p$ into $q$. For homogeneous 3-dimensional manifolds ($n=3$) there are three possibilities for the degree of rigidity, since they may have the isometry group of dimension $6$, $4$ or $3$. The maximum rigidity, $6$, corresponds to 3-dimensional space forms $M^3(\rho)$.

Applying Corollary \ref{helices}, for proper triharmonic curves in a 3-dimensional space form $M^3(\rho)$, allows us to state the following proposition.

\begin{prop}\label{pro:constantktau} Let $\gamma(s)$ be a proper triharmonic curve immersed in a 3-dimensional space form $M^3(\rho)$, then $\gamma(s)$ has constant curvature if and only if it has constant torsion.
\end{prop}
\begin{proof} We just need to prove that a triharmonic curve in $M^3(\rho)$ with constant torsion has also constant curvature. Let $\gamma(s)$ be a proper triharmonic curve with constant torsion $\tau(s)=\tau_o$. If $\tau_o=0$, we can assume that the curve $\gamma(s)$ lies on a totally geodesic surface of $M^3(\rho)$, that is, on $M^2(\rho)$. Then, from Corollary \ref{M2rho} we end the proof.\\
Therefore, we assume that $\tau(s)=\tau_o\neq 0$. By contradiction, we suppose that the curvature of $\gamma(s)$, $\kappa(s)$, is not constant. Using the Frenet equations \eqref{Frenetequations} the tangent and binormal components of \eqref{3-tension} become
\begin{eqnarray}
\kappa\kappa'''-\left(2\kappa^2+\tau_o^2\right)\kappa\kappa'+2\kappa'\kappa''&=&0\,,\label{help11}\\
4\kappa'''-4\tau_o^2\kappa'-9\kappa^2\kappa'+2\rho\kappa'&=&0\,.\label{help22}
\end{eqnarray}
Observe that, following the same method of \S 3 (see also \cite{K}), equation \eqref{help11} can be integrated twice obtaining (compare with \eqref{fi} for the case $\tau_o=0$)
\begin{equation}\label{*}
5\left(\kappa'\right)^2=c_2-2c_1\frac{1}{\kappa}+\kappa^4+\frac{5}{3}\tau_o^2\kappa^2\,,
\end{equation}
for some real constants $c_1$ and $c_2$.\\
On the other hand, if we multiply equation \eqref{help22} by $\kappa$ and combine it with \eqref{help11} to eliminate the term $\kappa'''$, we reach to an exact differential equation whose first integral is
\begin{equation}\label{**}
4\left(\kappa'\right)^2=c_o+\kappa^2\left(\rho-\frac{1}{4}\kappa^2\right),
\end{equation}
for a real constant $c_o$.\\
Finally, combining  \eqref{*} and \eqref{**}, we get the following  polynomial equation of degree five in $\kappa$,
$$21\kappa^5+\left(\frac{80}{3}\tau_o^2-20\rho\right)\kappa^3+\left(16c_2-20c_0\right)\kappa-32c_1=0\,,$$
which contradicts the assumption that $\kappa(s)$ is not constant.
\end{proof}

Now, if $\gamma(s)$ is a Frenet helix in $M^3(\rho)$, since $R^M\left(B,N\right)T{=R^M\left(T,B\right)N}=0$, equations \eqref{eq1helix} and \eqref{eq2helix} simplify to
$$\left(\kappa_o^2+\tau_o^2\right)^2=\left(2\kappa_o^2+\tau_o^2\right)\rho\,.$$
Hence, using the latter, we conclude with the following classification of triharmonic Frenet helices in 3-dimensional space forms.

\begin{thm}\label{4.4} Let $M^3(\rho)$ be a 3-dimensional space form and consider a Frenet helix $\gamma(s)$ immersed in $M^3(\rho)$. If $\gamma(s)$ is a triharmonic curve, then either it is a geodesic or $M^3(\rho)=\mathbb{S}^3(\rho)$ and the constant curvature of $\gamma(s)$ is given by
$$\kappa^2(s)=\kappa_o^2={\left(\rho-\tau_o^2\right)\pm\sqrt{\rho\left(\rho-\tau_o^2\right)}}\,,$$
where $\tau_o$ is the constant torsion of $\gamma(s)$. In particular, if $\tau_o=0$, we have that $\gamma(s)$ is a circle in $\mathbb{S}^2(\rho)$ satisfying $\kappa_o^2=2\rho$.
\end{thm}

We focus now on homogeneous 3-dimensional spaces with the isometry group of dimension $4$. These spaces include, amongst its simply connected members, the product spaces $\mathbb{S}^2(\rho)\times\mathbb{R}$ and $\mathbb{H}^2(\rho)\times\mathbb{R}$; the Berger spheres; the Heisenberg group; and the universal covering of the special linear group $Sl(2,\mathbb{R})$. Cartan in \cite{Ca} showed that all homogeneous 3-manifolds with the isometry group of dimension $4$ can be described by a Bianchi-Cartan-Vranceanu (BCV) space $M(a,b)$, where $4\,a\neq b^2$. We recall that  BCV spaces (see \cite{Bianchi,Ca,Vr}) are described by  the following two-parameter family of Riemannian metrics
\begin{equation}\label{1.1}
g_{a,b} =\frac{dx^{2} + dy^{2}}{[1 + a(x^{2} + y^{2})]^{2}} +  \left(dz + \frac{b}{2} \frac{ydx - xdy}{[1 + a(x^{2} + y^{2})]}\right)^{2},\quad a,b \in {\mathbb{R}}
\end{equation}
defined on $M^3=\{(x,y,z)\in\mathbb{R}^3\colon \lambda_a=1+a\left(x^2+y^2\right)>0\}$. We are going to denote these BCV spaces by $M(a,b)$, while the metrics $g_{a,b}$, simply, by $\langle,\rangle$.

Now, if we consider the orthonormal basis of vector fields given by $\{E_1,E_2,E_3\}$ where
\begin{equation}
E_1=\lambda_a\, \frac{\partial}{\partial x}-\frac{b\, y}{2}\frac{\partial}{\partial z}\, , \quad E_2=\lambda_a\,\frac{\partial}{\partial y}+\frac{b\, x}{2}\frac{\partial}{\partial z}\, ,\quad E_3=\frac{\partial}{\partial z}\, ,\label{orthonormalframe}
\end{equation}
we can write the expressions for the Levi-Civita connection as
\begin{equation}\label{eq:levi-civita-BCV}
\begin{array}{lll}
\nabla_{E_1} E_1=2\,a\, y\, E_2\, , \quad & \nabla_{E_1} E_2=-2\, a\, y\, E_1+\frac{b}{2}E_3\, ,\quad & \nabla_{E_1}E_3=-\frac{b}{2}E_2 \, ,  \\
\nabla_{E_2} E_1=-2\,a\,y\, E_1+\frac{b}{2}E_3\, ,\quad & \nabla_{E_2}E_2=2\,a\, x\, E_1\, ,\quad & \nabla_{E_2}E_3=\frac{b}{2}E_1\, ,  \\
\nabla_{E_3}E_1=-\frac{b}{2}E_2\, , \quad & \nabla_{E_3}E_2= \frac{b}{2}E_1\, ,\quad & \nabla_{E_3}E_3 =0\, . 
\end{array}
\end{equation}
Moreover, the nonzero components of the curvature tensor can be computed, obtaining
\begin{equation}
R_{1212}=4\,a-\frac{3}{4} b^2\, , \quad\quad R_{1313}=R_{2323}= \frac{b^2}{4}\, . \label{R}
\end{equation}

Observe that, from the above expressions of curvature tensor, if $4\,a=b^2$ then $M(a,b)$ represents a 3-dimensional space form. Therefore, from now on, we are going to assume that $4\,a\neq b^2$. In these cases, as mentioned before, the family of metrics \eqref{1.1} includes all three-dimensional homogeneous metrics whose isometry group has dimension $4$. The classification of these spaces is as follows 
\begin{itemize}
\item If $a=0$ and $b\neq 0$, we have that $M(a,b)\cong \mathbb{H}_3$, the \emph{Heisenberg group}.
\item If $a>0$ and $b=0$, $M(a,b)\cong\left(\mathbb{S}^2(4\,a) -\{\infty\}\right) \times \mathbb{R}$.
\item If $a<0$ and $b=0$, $M(a,b)\cong \mathbb{H}^2(4\,a)\times \mathbb{R}$.
\item If $a>0$, $b\neq 0$ and $4\,a\neq b^2$, then $M(a,b)\cong SU(2)- \{\infty\}$.
\item And, finally, if $a<0$ and $b\neq 0$, we have that $M(a,b)\cong \widetilde{Sl}(2,\mathbb{R})$.
\end{itemize}

The Lie algebra of the infinitesimal isometries of $M(a,b)$ with $4\,a\neq b^2$ admits the following basis of Killing vector fields
\begin{equation}\nonumber
\begin{array}{lll}
X_1&=&\left(1-\frac{2\,a\, y^2}{\lambda_a}\right) E_1+\frac{2axy}{\lambda_a} E_2+ \frac{b y}{\lambda_a} E_3\, , \\
X_2&=&\frac{2axy}{\lambda_a} E_1+\left(1-\frac{2ax^2}{\lambda_a}\right) E_2-\frac{bx}{\lambda_a} E_3\, ,\\
X_3&=&-\frac{y}{\lambda_a} E_1+\frac{x}{\lambda_a}E_2-\frac{b\left(x^2+y^2\right)}{2\lambda_a} E_3\, ,\\
X_4&=&E_3\, , 
\end{array}
\end{equation}
where $\{E_i\}$, $i=1,2,3$, is the orthonormal basis introduced in \eqref{orthonormalframe}. 

Then, a surface which stays invariant under the action of any Killing vector field, $\xi$, is called an \emph{invariant surface}. In particular, invariant surfaces under the action of the Killing vector field $X_4$ are usually called \emph{Hopf cylinders}. These cylinders can be parametrized as $\mathbf{x}(s,t)=\psi_t(\widetilde{\alpha}(s))$, where $\widetilde{\alpha}(s)$ denotes an arc-length parametrized curve orthogonal to $X_4$ in $M(a,b)$ while $\{\psi_t\,;\, t\in\mathbb{R}\}$ is the one-parameter group of isometries associated to $X_4$. 

Let $\gamma(s)$ be an arc-length parametrized triharmonic curve with constant curvature $\kappa(s)=\kappa_o\neq 0$, immersed in a BCV space $M(a,b)$ with $4a\neq b^2$. Then, by Corollary \ref{helices}, we have that the torsion of $\gamma(s)$ is also constant, that is $\gamma(s)$ is a Frenet helix. A partial converse of Corollary \ref{helices} holds in these spaces.

\begin{prop}\label{45} Let $\gamma(s)$ be an arc-length parametrized triharmonic curve with vanishing torsion immersed in a BCV space $M(a,b)$ with $4a\neq b^2$. Then, the curvature of $\gamma(s)$ is constant. 
\end{prop}
\begin{proof} Let $\gamma(s)$ denotes a triharmonic curve with $\tau(s)=0$. We first note that, since the torsion vanishes, the binormal $B$ is constant along $\gamma$. In fact,
$$\nabla_T B(s)=-\tau(s)N(s)=0$$
holds from \eqref{Frenetequations}. Hence, in particular, $B_3=\langle B,E_3\rangle$ is constant along $\gamma$ and so is
$$\langle R^M\left(N,T\right)T,N\rangle=\frac{b^2}{4}+\left(4a-b^2\right)B_3^2\,.$$
Finally, using that $\langle R^M\left(N,T\right)T,N\rangle$ is constant and $\tau(s)=0$, a similar argument as in Theorem \ref{all} concludes the proof.
\end{proof}
Consider a proper triharmonic Frenet helix $\gamma(s)$ with (constant) curvature $\kappa_o$ and torsion $\tau_o$. Then, equations \eqref{eq1helix} and \eqref{eq2helix} must hold and, using \eqref{R}, they become
\begin{eqnarray}
\left(\kappa_o^2+\tau_o^2\right)^2-\left(2\kappa_o^2+\tau_o^2\right)\left(\frac{b^2}{4}+\left[4a-b^2\right]B_3^2\right)-\kappa_o\tau_o\left(4a-b^2\right)T_3B_3&=&0\,, \label{eq1cst}\\
{\left(2\kappa_o^2+\tau_o^2\right)N_3B_3+\kappa_o}\tau_oT_3N_3&=&0\,,\label{eq2cst}
\end{eqnarray}
where $T_3=\langle T,E_3\rangle$, $N_3=\langle N, E_3\rangle$ and $B_3=\langle B, E_3\rangle$.

{If the constant torsion $\tau_o$ is identically zero, then \eqref{eq2cst} becomes $N_3B_3=0$}, while \eqref{eq1cst} simplifies to
\begin{equation}\label{ko}
\kappa_o^2=2\left(\frac{b^2}{4}+\left[4a-b^2\right]B_3^2\right)
\end{equation}
if $\kappa_o\neq 0$. Note that in this case, since $\tau_o=0$, $\nabla_T B(s)=0$ holds and, hence, $B_3$ is a constant along the Frenet helix. The existence of proper triharmonic helices with vanishing torsion depends on the value of the constant $B_3$. In fact, with the aid of  Proposition~\ref{45}, we have immediately the following result.

\begin{thm}\label{46} Let $\gamma(s)$ be an arc-length parametrized triharmonic curve with vanishing torsion immersed in a BCV space $M(a,b)$ with $4a\neq b^2$. Then, either $\gamma(s)$ is a geodesic or $\gamma(s)$ is a Frenet helix where its constant curvature $\kappa_o$ is given by \eqref{ko}. Moreover, $B_3=\langle B, E_3\rangle$ is a constant satisfying 
$$B_3^2<\frac{b^2}{4\left(b^2-4a\right)}$$
if $b^2>4a$; or, $B_3\neq 0$ if $b^2<4a$. In particular, there are no proper triharmonic curves with vanishing torsion in the product space $\mathbb{H}^2(4a)\times\mathbb{R}$.
\end{thm}

Next, we focus on triharmonic helices with nonzero constant torsion. We need the following technical lemma.

\begin{lem}\label{lem} Let $\gamma(s)$ be a non-geodesic curve parametrized by arc-length immersed in a BCV space $M(a,b)$ with $4a\neq b^2$. Then $T_3=\langle T, E_3\rangle$ is constant if and only if $N_3=\langle N, E_3\rangle=0$. 
\end{lem}
\begin{proof} For the arc-length parametrized curve $\gamma(s)$, we write its unit tangent vector field $T(s)$ with respect to the orthonormal frame $\{E_i\}$, $i=1,2,3$, introduced in \eqref{orthonormalframe}.
Then, with the aid of \eqref{eq:levi-civita-BCV}, we compute (for details see \cite[Lemma 5.5]{CMOP})
$$\langle\nabla_T T,E_3\rangle=\frac{d}{ds}\langle T,E_3\rangle=T_3'(s)=\kappa(s)\langle N,E_3\rangle=\kappa(s)N_3\,,$$
where $\kappa(s)\neq 0$ is the curvature of $\gamma(s)$.
We conclude that $T_3$ is constant if and only if $N_3=0$. \end{proof}

{Now, suppose that $N_3\neq 0$. In this case, from Lemma \ref{lem} and the relation $B_3'=-\tau_oN_3\neq 0$, both $T_3$ and $B_3$ are nonconstant functions and equation \eqref{eq2cst} reads
$$\left(2\kappa_o^2+\tau_o^2\right)B_3+\kappa_o\tau_oT_3=0\,.$$
Differentiating this equation, we conclude that $\tau_o=0$, since $N_3\neq 0$, which contradicts the assumption that the curve is a triharmonic helix with nonzero constant torsion.}

Therefore, for a proper triharmonic Frenet helix with nonzero torsion, equation \eqref{eq2cst} is satisfied if and only if $N_3=0$. Moreover, Frenet helices satisfying $N_3=0$ are geodesics of Hopf cylinders as proved in \cite{BFG}. We thus have

\begin{cor}\label{hopf} Let $\gamma(s)$ be an arc-length parametrized triharmonic Frenet helix immersed in a BCV space $M(a,b)$ with $4a\neq b^2$. If the torsion of  $\gamma(s)$ is not zero, then $\gamma(s)$ is a geodesic of a suitable Hopf cylinder.
\end{cor}

In the final part of this section we shall give the explicit parametrizations of triharmonic helices. Assume that $\gamma(s)$ is a non-geodesic arc-length parametrized curve immersed in a BCV space $M(a,b)$ with $4a\neq b^2$ and satisfying that $N_3=0$. Then, following the computations of \cite[\S 5.2]{CMOP}, we have that the curvature and the torsion of $\gamma(s)$ are given by
\begin{eqnarray}
\kappa(s)&=&\zeta \sin\alpha_o\,,\label{ks}\\
\tau(s)&=&-\zeta\cos\alpha_o-\frac{b}{2}\,,\label{ts}
\end{eqnarray}
where $\alpha_o\in\left(0,\pi\right)$ is a constant and 
\begin{equation}\label{beta(s)}
\zeta=\beta'(s)+2a\sin\alpha_o\left[y\cos\beta(s)-x\sin\beta(s)\right]-b\cos\alpha_o>0
\end{equation}
for some function $\beta(s)$. Moreover, the Frenet frame along $\gamma$ with respect to the orthonormal frame \eqref{orthonormalframe} is given by
\begin{equation}\label{frenetN3=0}
\begin{cases}
T(s)=\sin\alpha_o\cos\beta(s)E_1+\sin\alpha_o\sin\beta(s)E_2+\cos\alpha_o E_3 \,,\\
N(s)=-\sin\beta(s)E_1+\cos\beta(s)E_2 \,,\\
B(s)=T(s)\times N(s)=-\cos\alpha_o\cos\beta(s)E_1-\cos\alpha_o\sin\beta(s)E_2+\sin\alpha_o E_3 \,.
\end{cases}
\end{equation}

If we also require that $\gamma(s)$ is a Frenet helix, then $\zeta$ is constant. Furthermore, substituting  the above data in \eqref{eq1cst} we conclude that for a proper triharmonic Frenet helix  the constant $\zeta$ must be a positive root of the four degree polynomial
\begin{eqnarray}
P_4(\zeta)&=&4 \zeta^4+8 b \cos\alpha_o \zeta^3+\left(5b^2\cos^2\alpha_o-8\left[4a-b^2\right]\sin^4\alpha_o\right)\zeta^2\nonumber\\&+&b\left(b^2-2\left[4a-b^2\right]\sin^2\alpha_o\right)\cos\alpha_o \zeta-b^2\left(4a-b^2\right)\sin^2\alpha_o\,.\label{polynomial}
\end{eqnarray}

When $a\neq 0$, the parametrization of Frenet helices in $M(a, b)$ satisfying $N_3=0$ was given in \cite[Lemma~2]{GM} (which is an adapted version of \cite[Theorem~5.6]{CMOP}). Using these results we have immediately the following explicit description.

\begin{thm}\label{others} Let $\gamma(s)$ be an arc-length parametrized triharmonic curve with constant curvature, $\kappa_o$, immersed in a BCV space $M(a, b)$ with $4a\neq b^2$ and $a\neq 0$. Let $\zeta$ be a positive root of the polynomial \eqref{polynomial}. Then, $\gamma(s)$ is either a geodesic ($\kappa_o=0$) or a Frenet helix parametrized by one of the following types:
\begin{enumerate}[(i)]
\item If $\beta(s)$ is  a nonconstant solution of \eqref{beta(s)},
\begin{eqnarray*}
x(s)&=&\mu\sin\alpha_o\sin\beta(s)+c_1\,,\\
y(s)&=&-\mu\sin\alpha_o\cos\beta(s)+c_2\,,\\
z(s)&=&\frac{b}{4a}\beta(s)+\frac{1}{4a}\left(\left[4a-b^2\right]\cos\alpha_o-b\zeta\right)s\,,
\end{eqnarray*}
where $\mu>0$ and $c_1$, $c_2$ are constants satisfying
$$c_1^2+c_2^2=\frac{\mu}{a}\left(\left[b\cos\alpha_o+\zeta-\frac{1}{\mu}\right]+a\mu\sin^2\alpha_o\right).$$
\item If $\beta(s)=\beta_o$ is a constant such that $\sin\beta_o\cos\beta_o\neq 0$,
\begin{eqnarray*}
x(s)&=&x(s)\,,\\
y(s)&=&x(s)\tan\beta_o+c_1\,,\\
z(s)&=&\frac{1}{4a}\left(\left[4a-b^2\right]\cos\alpha_o-b\zeta\right)s+c_2\,,
\end{eqnarray*}
where $c_2\in\mathbb{R}$, the constant $c_1$ is given by
$$c_1=\frac{\zeta+b\cos\alpha_o}{2a\sin\alpha_o\cos\beta_o}$$
and $x(s)$ is a solution of the ordinary differential equation
$$x'(s)=\left(1+a\left[x^2(s)+\left(x(s)\tan\beta_o+c_1\right)^2\right]\right)\sin\alpha_o\cos\beta_o\,.$$
\item If $\beta(s)=\beta_o$ is a constant satisfying $\sin\beta_o\cos\beta_o=0$ (up to interchange of $x$ with $y$),
\begin{eqnarray*}
x(s)&=&x_o=\mp \frac{\zeta+b\cos\alpha_o}{2a\sin\alpha_o}\,,\\
y(s)&=&y(s)\,,\\
z(s)&=&\frac{1}{4a}\left(\left[4a-b^2\right]\cos\alpha_o-b\zeta\right)s+c_1\,
\end{eqnarray*}
for a constant $c_1\in\mathbb{R}$ and where $y(s)$ is a solution of the ordinary differential equation
$$\left(y'(s)\right)^2=\left(1+a\left[x_o^2+y^2(s)\right]\right)^2\sin^2\alpha_o\,.$$
\end{enumerate}
\end{thm}

In the particular case that $b=0$, the polynomial  \eqref{polynomial} reduces to
$$P_4(\zeta)\lvert_{b=0}=4\left(\zeta^2-8a \sin^4\alpha_o\right)\zeta^2\,.$$
Since we are seeking positive roots, necessarily $\zeta^2=8a\sin^4\alpha_o$ holds, which implies that $a>0$. Therefore, we have the following consequence.

\begin{cor} There are no proper triharmonic curves with constant curvature immersed in the product space $\mathbb{H}^2(4a)\times\mathbb{R}$.
\end{cor}

To end this section, we consider the case $M(a,b)$ with $a=0$ and $b\neq 0$, which is not included in Theorem~\ref{others}. This case corresponds to the Heisenberg group $\mathbb{H}_3$. We recall that $\mathbb{H}_3$ can be seen as the  Lie group $\left(\mathbb{R}^3,*\right)$ where $*$ is defined by
$$\left(x_1,y_1,z_1\right)*\left(x_2,y_2,z_2\right)=\left(x_1+x_2,\,y_1+y_2,\,z_1+z_2+b\left[x_1y_2-y_1x_2\right]\right)\,,$$
for $\left(x_i,y_i,z_i\right)\in\mathbb{R}^3$, $i=1,2$.

We now derive the explicit parametrizations of triharmonic Frenet helices in $\mathbb{H}_3$.

\begin{thm}\label{h3} Let $\gamma(s)$ be an arc-length parametrized triharmonic curve with constant curvature, $\kappa_o$, immersed in the Heisenberg group $\mathbb{H}_3$. If $\gamma(s)$ is not a geodesic ($\kappa_o\neq 0$), then it is a Frenet helix parametrized (up to left translations) by
\begin{eqnarray*}
x(s)&=&\frac{\sin\alpha_o}{\zeta+b\cos\alpha_o}\left(\sin\beta(s)-\sin\lambda\right),\\
y(s)&=&\frac{-\sin\alpha_o}{\zeta+b\cos\alpha_o}\left(\cos\beta(s)-\cos\lambda\right),\\
z(s)&=&\frac{\left(2\zeta+b\cos\alpha_o\right)\cos\alpha_o+b}{2\left(\zeta+b\cos\alpha_o\right)}s+\frac{b\sin^2\alpha_o}{2\left(\zeta+b\cos\alpha_o\right)^2}\left(\sin\lambda\cos\beta(s)-\cos\lambda\sin\beta(s)\right),
\end{eqnarray*}
where $\beta(s)=\left(\zeta+b\cos\alpha_o\right)s+\lambda$ ($\lambda\in\mathbb{R}$), and $\zeta$ is a positive root of the polynomial $P_4(\zeta)$, \eqref{polynomial}, with $a=0$.
\end{thm}
\begin{proof} We assume that $\gamma(s)=\left(x(s),y(s),z(s)\right)$ is a non-geodesic arc-length parametrized curve with constant curvature $\kappa_o\neq 0$ in the Heisenberg group $\mathbb{H}_3$. Since the curvature is constant and $\gamma(s)$ is triharmonic, by Corollary \ref{helices}, $\gamma(s)$ is a Frenet helix. Moreover, the triharmonic condition also implies that $N_3=0$ holds. Then the curvature and the torsion of $\gamma(s)$ are given by \eqref{ks} and \eqref{ts}, respectively, where $\zeta$ is any positive root of $P_4(\zeta)$, \eqref{polynomial}, for $a=0$. In particular, since $\zeta$ is constant, integrating \eqref{beta(s)} we obtain
$$\beta(s)=\left(\zeta+b\cos\alpha_o\right)s+\lambda$$
for some constant $\lambda$.\\
At the same time, the Frenet frame along $\gamma(s)$ is described in \eqref{frenetN3=0}. Hence, we just need to solve the system of ordinary differential equations
\begin{eqnarray*}
x'(s)&=&\sin\alpha_o\cos\beta(s)\,,\\
y'(s)&=&\sin\alpha_o\sin\beta(s)\,,\\
z'(s)&=&\cos\alpha_o+\frac{b}{2}\sin\alpha_o\left(x(s)\sin\beta(s)-y(s)\cos\beta(s)\right).
\end{eqnarray*}
Finally, since in $\mathbb{H}_3$ it is enough to obtain the parametrizations of tiharmonic curves starting at $\left(0,0,0\right)$ and then use left translations to move them around, we can integrate the above system with the initial condition $\gamma(0)=\left(0,0,0\right)$ and conclude the proof.
\end{proof}

\end{document}